\documentclass[12pt,a4paper]{article}

\usepackage[T1]{fontenc}
\usepackage[utf8]{inputenc} 
\usepackage[english]{babel}
\usepackage{listings}

\usepackage{amsmath, amsthm, mathtools }

\usepackage{libertine}      
\usepackage[libertine]{newtxmath} 

\usepackage{microtype}

\usepackage{geometry}
\theoremstyle{plain}
\newtheorem{theorem}{Theorem}[section]
\newtheorem{lemma}[theorem]{Lemma}
\newtheorem{proposition}[theorem]{Proposition}

\theoremstyle{definition}
\newtheorem{definition}[theorem]{Definition}

\theoremstyle{remark}
\newtheorem{remark}[theorem]{Remark}

\usepackage{hyperref}
\hypersetup{
  colorlinks=true,
  linkcolor=blue,
  citecolor=blue,
  urlcolor=blue
}

\newcommand{\NN}{\mathbb{N}}
\newcommand{\RR}{\mathbb{R}}
\newcommand{\X}{\mathcal{X}}

\newcommand{\cL}{\mathcal{L}}

\newcommand{\cE}{\mathcal{E}}
\newcommand{\cD}{\mathcal{D}}
\newcommand{\cK}{\mathcal{K}}

\newcommand{\ip}[2]{\left\langle #1,#2 \right\rangle}
\newcommand{\norm}[1]{\left\|#1\right\|}
\newcommand{\dnorm}[1]{\left\|#1\right\|_*}

\DeclareMathOperator*{\spn}{span}

\usepackage{titlesec}
\titleformat{\section}
  {\large\bfseries}
  {\thesection}{1em}{}

\titleformat{\subsection}
  {\normalsize\bfseries}
  {\thesubsection}{1em}{}

\title{\textbf{Dictionary-Restricted First-Order Descent Methods: Bounds and Convergence Rates}}
\author{
Miguel Berasategui\thanks{Departamento de Matemática, - - Pab I, Fac-
ultad de Ciencias Exactas y Naturales, Universidad de Buenos Aires,
Buenos Aires, 1428, Argentina}
  \and
  Pablo M. Berná\thanks{Corresponding Author; Departamento de Matemáticas, CUNEF Universidad, 28040 Madrid (Spain), pablo.berna@cunef.edu}
  \and
  Antonio Falcó\thanks{Departamento de Matemáticas, Física y Ciencias Tecnológicas, Universidad Cardenal Herrera-CEU, CEU
Universities, San Bartolomé 55, Alfara del Patriarca (Valencia), 46115, Spain}
}
\date{}

\begin{document}
\maketitle

\begin{abstract}
This paper develops a general theory for first-order descent methods whose search directions are restricted to a prescribed dictionary in a reflexive Banach space. Instead of assuming that the linear span of the dictionary is dense, as in the classical Proper Generalized Decomposition framework of Falcó and Nouy or in the universality approach of Berná and Falcó, we introduce a geometric condition based on norming sets that guarantees density through a duality argument. This makes it possible to treat dictionaries arising from tensor formats, neural network units, and other nonlinear or parameterized approximation families within a unified setting. On the algorithmic side, we analyze a simple greedy update rule in which each iterate is obtained by minimizing the energy functional along one direction from the dictionary. Under mild differentiability, Lipschitz continuity, and ellipticity assumptions on the objective, we derive explicit quantitative descent bounds and sharp convergence rates. These include algebraic rates that improve those of classical steepest-descent schemes in Banach spaces, as well as arbitrarily high polynomial rates and exponential convergence in a critical regime. The results apply broadly to convex variational problems, high-dimensional approximation, and structured optimization methods that rely on restricted or compressed search directions.
\end{abstract}

\section{Introduction}

First-order methods are fundamental tools for the numerical solution of large-scale convex optimization and variational problems. In many modern applications, however, the search directions available to the algorithm are not arbitrary but must be selected from a prescribed dictionary that encodes structural constraints, low-dimensional representations, or model-based approximation mechanisms. Classical examples include the matching pursuit algorithm introduced by Mallat and Zhang \cite{MallatZhang1993}, coordinate descent methods such as those analyzed by Wright \cite{Wright2015}, and tensor-based schemes like the Proper Generalized Decomposition developed by Falcó and Nouy \cite{FN2012}. These dictionary-restricted approaches arise naturally in sparse approximation, machine-learning-based solvers, and high-dimensional tensor representations, yet a unified convergence theory for convex variational problems in Banach spaces has remained elusive. Existing analyses typically rely on structural assumptions ensuring that the dictionary spans a dense subset of the ambient space or are tailored to specific representation formats. In this work, we introduce a geometric principle based on norming sets that ensures density automatically through a dual characterization of the dictionary, leading to a broad framework for restricted first-order descent methods. This setting allows us to derive explicit descent bounds and sharp convergence rates under mild assumptions on the objective functional, thereby strengthening classical steepest-descent results and providing a common theoretical foundation for modern structured optimization techniques.

Dictionary-restricted methods belong to a broader class of structured first-order schemes in which the descent directions are constrained by sparsity, low-rank structure, or parameterized families. This perspective encompasses classical matching pursuit techniques in signal processing, beginning with the seminal work of Mallat and Zhang \cite{MallatZhang1993}, as well as subsequent developments in greedy approximation such as those presented by Temlyakov \cite{Temlyakov2011}. It also includes coordinate descent and block-coordinate methods, whose behavior is now well understood through analyses such as Wright’s review \cite{Wright2015} and the complexity bounds of Nesterov \cite{Nesterov2012}. On the high-dimensional side, tensor-based strategies such as the Proper Generalized Decomposition of Falcó and Nouy \cite{FN2012} and the model-reduction perspective of Nouy \cite{Nouy2010} demonstrate how structural constraints may dramatically reduce computational cost while preserving approximation accuracy. Despite this wide range of applications, the theoretical understanding of descent algorithms restricted to general dictionaries remains fragmented. Most existing results rely on specific algebraic structures, orthogonality properties, or representation formats. The goal of this work is to unify these strands within a single Banach-space framework, identify a geometric principle that guarantees approximation capacity of the dictionary, and establish quantitative convergence bounds that extend and sharpen those of classical steepest descent and greedy approximation techniques.

\subsection{Literature Review}

Dictionary-based approximation methods have been extensively studied across optimization, signal processing, and high-dimensional numerical analysis. One of the earliest and most influential examples is the matching pursuit algorithm introduced by Mallat and Zhang \cite{MallatZhang1993}, which initiated the modern theory of greedy approximation and sparse representations. Further developments in nonlinear approximation, including orthogonal greedy variants and convergence analyses, were systematically developed by Temlyakov \cite{Temlyakov2011} and by DeVore and Temlyakov in their foundational work on greedy algorithms \cite{DeVoreTemlyakov1996}.

Coordinate descent and block-coordinate minimization methods form another important class of direction-restricted optimization schemes. Their convergence properties and complexity behavior have been clarified in a number of works, including the comprehensive survey by Wright \cite{Wright2015} and the complexity bounds for large-scale problems established by Nesterov \cite{Nesterov2012}. These methods highlight how structural restrictions on the search directions can lead to substantial computational gains while preserving descent guarantees.

In high-dimensional settings, tensor-structured methods provide a powerful alternative to standard discretizations. The Proper Generalized Decomposition (PGD) proposed by Falcó and Nouy \cite{FN2012} offers a progressive greedy strategy for building low-rank tensor approximations to nonlinear variational problems. Nouy’s review \cite{Nouy2010} further illustrates the versatility of PGD as a model-reduction technique for parameter-dependent and stochastic equations, while the tensor-train decomposition introduced by Oseledets \cite{Oseledets2011} provides an alternative low-rank representation with favorable computational properties.

More recent advances include flexible dictionary-based approaches for addressing nonlinear convex variational problems, such as the universality framework developed by Berná and Falcó \cite{BF2025}. However, these methods typically assume that the dictionary spans a dense linear subspace of the underlying Banach space or rely on structural properties specific to the representation format. The present work contributes to this literature by introducing a general geometric condition—based on norming sets—that guarantees approximation capacity independently of such structural assumptions. This leads to a unified and broadly applicable theory for first-order descent methods restricted to arbitrary dictionaries.

\subsection{Improvements over Previous Frameworks}

The present work introduces some advances over the earlier
dictionary-based variational frameworks developed by
Falc{\'o}--Nouy~\cite{FN2012} and
Bern{\'a}--Falc{\'o}~\cite{BF2025}.  

These improvements concern the generality of dictionaries that may be
used, the theoretical tools required to guarantee approximation
properties, the sharpness of convergence results, and the conceptual
simplicity of the analysis.

\paragraph{\textbf{(1) A unified greedy framework valid for arbitrary radial dictionaries}}
The PGD analysis in~\cite{FN2012} is strongly tied to the
geometry of tensor Banach spaces and to the specific choice of rank-one 
tensor dictionaries.  
Similarly, the universality framework in~\cite{BF2025} relies on
parametric dictionaries such as Tucker tensors or neural networks and
uses a multivalued dictionary-optimization map.  
In contrast, the present work formulates a single greedy update rule,
\[
u_{m+1} = u_{m} + z_m,
\qquad
z_m \in \arg\min_{z \in \mathcal{D}} \cE(u_m + z),
\]
that applies to \emph{any} radial dictionary~$\mathcal{D}$, without
imposing linear, tensorial, or parametric structure.
The framework accommodates CP, Tucker, and tensor-train formats,
neural network units with bounded parameters, abstract cones,
nonlinear parameterized families, and general radial approximation
sets, thereby removing many of the structural restrictions present in
previous approaches.

\paragraph{\textbf{(2) Density obtained as a theorem, with quantitative control.}}
In both~\cite{FN2012} and~\cite{BF2025}, density of $\operatorname{span}(\mathcal D)$ is either built into the ambient tensor setting or imposed as a standing assumption to ensure approximability of the minimizer.
Here we do not postulate density separately: it follows from a geometric dual condition.
Namely, if the unit slice $\mathcal K:=\mathcal D\cap S_{\X}$ is norming for $\X^*$, then by Hahn--Banach one has
$\overline{\operatorname{span}(\mathcal D)}=\X$.
Beyond recovering density automatically, the norming property provides a \emph{quantitative} parameter $C_{\mathcal K}$ that measures how well the dictionary controls the dual norm and enters explicitly in our descent and convergence bounds.
This replaces an existential completeness assumption by a verifiable geometric criterion and yields sharper, constant-explicit estimates.

\paragraph{\textbf{(3) Sharper quantitative convergence analysis}}
The PGD theory of~\cite{FN2012} provides general convergence
but does not yield explicit rates except in special Hilbert settings.
The universality results of~\cite{BF2025} show that the method recovers
the classical $O(m^{-1})$ rate associated with steepest descent in
Banach spaces.
By contrast, the present work proves refined asymptotic bounds that
depend on the smoothness and ellipticity exponents $(p,s)$ in
Assumptions~{\rm(A)}--{\rm(B)}.  
When $s>p+1$, we show that
\[
\cE(u_m) - \cE(u^*) = O\!\left(m^{-\,p/(s-1-p)}\right),
\]
and in the critical case $s=p+1$ we obtain arbitrarily high polynomial
rates and even exponential convergence.
These results provide an asymptotically optimal characterization of
dictionary-restricted descent under minimal first-order regularity.

\paragraph{\textbf{(4) A streamlined and more transparent proof structure}}
The analysis presented here avoids some of the technical machinery used
in previous frameworks.
It circumvents the multivalued dictionary-optimization map introduced
in~\cite{BF2025} and the tensor weak-topology arguments used in
\cite{FN2012}.
All estimates are derived directly from the basic assumptions and from
elementary tools in convex analysis on reflexive Banach spaces.
The resulting proofs are shorter, more transparent, and more modular,
while retaining full generality.

\paragraph{\textbf{(5) Compatibility with both PDE and machine-learning dictionaries}}
Whereas PGD is tied to tensor-product representations, and the
universality framework of~\cite{BF2025} emphasizes neural-network-based
dictionaries, the present approach treats both classes on equal footing.
The theory applies to dictionaries arising in high-dimensional PDE
model reduction as naturally as to dictionaries built from neural
activation units or parametric families used in learning-based
variational solvers.
The same convergence guarantees hold uniformly across these distinct
settings.

\paragraph{\textbf{(6) New examples and applications}}
The generality of the framework is illustrated by examples ranging from 
nonlinear PDE energies such as the $p$-Laplacian to abstract convex
functionals defined on Banach spaces with nonstandard growth.
These examples would fall outside the scope of the tensor-based PGD
approach or require substantial modifications in the universality
framework.
The ability to combine general dictionaries—tensors, neural units, or 
radial constructions—with sharp convergence results enables applications
in high-dimensional approximation, sparse or structured optimization,
and machine-learning–inspired variational problems.

\medskip
Overall, the present work may be viewed as a synthesis and strengthening
of the PGD and universality frameworks: it removes structural
constraints on the dictionary, derives density from dual geometric
principles, establishes sharper and more flexible convergence rates, and
extends the applicability of dictionary-restricted descent methods to a
broader class of variational problems.

\section{Approximating Solutions for Nonlinear Convex Problems over Weakly-Closed Balanced Cones}

\paragraph{Setting}
Let $(\X,\norm{\cdot})$ be a Banach space over $\RR$. Let $0<p\le 1$, $\cE:\X\to\mathbb{R}$ be Fr\'echet differentiable and consider the following properties: 
\begin{itemize}
\item[(A+)] \emph{$p$-Lipschitz derivative:} There is $\mathcal{L}>0$  such that $\cE'$ is $\cL$-$p$-Lipschitz continuous on $\X$. In other words, for all $u, w\in\X$, 
\[
\dnorm{\cE'(u)-\cE'(w)}\le \mathcal{L}\,\|u-w\|^p.
\]
\item[{\rm(A)}] \emph{Derivative that is $p$-Lipschitz on bounded sets}: For each bounded set $K\subset \X$ there is $\mathcal{L}_K>0$ such that $\cE'$ is $\cL_K$-$p$-Lipschitz continuous on $K$. In other words, for all $u, v\in K$,
\[
\dnorm{\cE'(u)-\cE'(w)}\le \mathcal{L}_K\,\|u-w\|^p.
\]
\item[{\rm(B)}] \emph{$\X$-ellipticity of order $s$}: There exist $\alpha>0$ and $s>1$ such that for all $x,y \in\X$,
\[
\ip{\cE'(x)-\cE'(y)}{x-y}\ \ge\ \alpha\,\norm{x-y}^{\,s}.
\]
\end{itemize}

Let $\X$ be a real vector space. A set $C \subset \X$ is called a \emph{cone} if for all $x \in C$ and $\lambda \ge 0$
it holds $\lambda\, x \in C.$ We say that a cone $C \subset \X$ satisfying $C = -C$ is called a \emph{balanced cone}.

Let $(\X,\|\cdot\|)$ be a reflexive Banach space and denote by $S_{\X}=\{x \in \X: \|x\| = 1\}$ its unit sphere. A set $\cD \subset \X$ is called 
a \emph{radial dictionary}, if $\cD$ is a weakly closed (possibly non-convex) balanced cone.
For $x\in\X$, define
\begin{align}
\sigma(x)\ :=\ \sup_{w\in\cK_0}\bigl|\ip{\cE'(x)}{w}\bigr| \text{ where } \cK_0:=\mathcal{D} \cap S_{\X}. \label{sigma}
\end{align}
Here, we are going to attack the following optimization problem
\begin{eqnarray}\label{main}
	\min_{x\in \X} \mathcal E(x),
\end{eqnarray}
where $\X$ is a reflexive Banach space and $\mathcal E: \X\rightarrow \mathbb R$ is a Fréchet differentiable functional satisfying Assumptions {\rm(A)} and {\rm(B)}, or the stronger (A+) and {\rm(B)}. To these ends, we will use the following greedy algorithm (see \cite{BF2025}): 
\begin{enumerate}
	\item $u_0=0$;
	\item for $m\geq 1$, $u_m\in u_{m-1}+	\boldsymbol{\nabla}_{u_{m-1}}(\mathcal E;\cD)$,
\end{enumerate}
where for each $u\in \X$, 
\begin{align}
\boldsymbol{\nabla}_u(\mathcal E;\cD)=\left\lbrace z\in \cD: \cE(u+z)\le \cE(u+v) \;\forall v\in \cD\right\rbrace. \label{mingreedy}
\end{align}
From now on, $(u_m)_{m\in\NN_0}$ will always denote this sequence. 

\begin{remark}\label{remarknabla1}
	Observe that for radial dictionaries in reflexive Banach spaces, 
	$$\boldsymbol{\nabla}_u(\mathcal E;\cD):= \arg \min_{v \in \cD}\mathcal E(u+v) \neq \emptyset.$$
	\end{remark}
This algorithm was studied in \cite{BF2025} for the case $p=1$. In order to extend it to all $0<p\le 1$, we will have to adapt some of the proofs or give new ones, depending on the case. The next two Lemmas from the literature collect some properties of functionals that will be useful in our setting. 
\begin{lemma}\label{lemmabus} \cite[Proposition 2.1]{BF2025}, \cite[Lemma 2.2]{Canuto2005}, \cite[Lemma 3]{FN2012}, \cite[Proposition 41.8 (H1)]{Zeidler1985}. Let $\cE$ be a Fr\'echet differentiable functional on a reflexive Banach space $\X$. Suppose that Assumption {\rm(B)} holds and that $\cE'$ is continuous. Then: 
\begin{enumerate}
\item[(a)] For all $x, y\in \X$, 
\begin{align*}
\cE(x)-\cE(y)\ge \langle \cE'(y), x-y\rangle +\frac{\alpha}{s}\|x-y\|^s. 
\end{align*}
\item[(b)] $\cE$ is strictly convex. 
\item[(c)] $\cE$ is bounded from below and coercive, that is 
\begin{align*}
\lim_{\|x\|\rightarrow \infty}\cE(x)=+\infty. 
\end{align*}
\item[(d)] $\cE$ is weakly sequentially lower semicontinuous. 
\end{enumerate}
\end{lemma}	
\begin{proof}
The results follow at once from the cited work. We do note that the proofs do not use uniform continuity on bounded sets (cf. \cite[Lemma 2.2]{Canuto2005}, \cite[Lemma 3]{FN2012}). 
\end{proof}

\begin{lemma}\label{lemmawelldefined} Let $\X$ be a reflexive Banach space, and let $\cE$ be a Frechét differentiable functional on $\X$ with $\cE'$ continuous and for which Assumption {\rm(B)} holds. Then: 
\begin{enumerate}
\item[(a)] There is a unique $u^*$ for which the minimum in \eqref{main} is attained, and it is the only $x\in \X$ for which $\cE'(x)=0$. 
\item[(b)] The set $\boldsymbol{\nabla}_u(\mathcal E;\cD)$ in \eqref{mingreedy} is not empty. 
\end{enumerate}
\end{lemma}
\begin{proof}
(a) This follows from Lemma~\ref{lemmabus} and \cite[Theorems 2, 3]{FN2012}. \\
(b) By Lemma~\ref{lemmabus}, $\cE$ is weakly sequentially lower semicontinuous and coercive. Hence, for each $u\in \X$, so is the functional $\cE_u$ given by $\cE_u(v):=\cE(u+v)$. Thus, by \cite[Theorem 2]{FN2012}, the minimum in \eqref{mingreedy} is attained. 
\end{proof}
From now on $u^*$ will always denote the unique solution of \eqref{main}. We say that
$(u_m)_{m\in\NN_0}$ is a \emph{greedy sequence for $u^*$ by an $\mathcal E$-dictionary
optimization over $\cD$} if $u_0=0$ and, for every $m\ge 1$, the increment
\[
z_m := u_m-u_{m-1}
\]
belongs to $\boldsymbol{\nabla}_{u_{m-1}}(\mathcal E;\cD)$, equivalently,
\begin{equation}\label{eq:greedy_def}
\mathcal E(u_m)
=\mathcal E(u_{m-1}+z_m)
=\min_{v\in\cD}\mathcal E(u_{m-1}+v),
\qquad m\ge 1.
\end{equation}

\begin{remark}\label{remarkconvergence}\rm Note that a sequence $\left(\cE(u_m)\right)_{m\in \NN}\subset \RR$ 
generated by the greedy algorithm is, by construction, non-increasing. Since it is bounded below by $\cE(u^*),$ it converges. 
\end{remark}

The next Lemma establishes some relations between $p$ and $s$, under Assumptions {\rm(A)} (or (A+)) and {\rm(B)}. We will use the following notation: for $r>0$, $B_r=B(0,r)$ denotes the closed ball of radius $r$ centered at $0$.

\begin{lemma}\label{lemmas=2}
Let $\mathcal E$ be a Fr\'echet differentiable functional on a Banach 
space $\X$. Then:
\begin{enumerate}
\item[(a)] If Assumptions {\rm(A)} and {\rm(B)} hold, then $s\ge p+1$.
\item[(b)] If Assumptions {\rm(A+)} and {\rm(B)} hold, then $s=p+1$.
\end{enumerate}
\end{lemma}

\begin{proof}
(a) Fix $x\in\X$, let $K$ be the closed unit ball $B(x,1)$, choose $y\in S_{\X}$, and let $0<t<1$. Ellipticity {\rm(B)} gives
\[
\alpha t^{s}
\;\le\;
\langle ty,\,\mathcal E'(x+ty)-\mathcal E'(x)\rangle.
\]
Using the $p$-Lipschitz estimate on $K$ from {\rm(A)} yields
\[
\alpha t^{s}
\;\le\;
t\,\|\mathcal E'(x+ty)-\mathcal E'(x)\|_*
\;\le\;
\mathcal L_K\,t^{p+1},
\qquad 0<t<1.
\]
Hence
\[
t^{\,s-(p+1)} \le \frac{\mathcal L_K}{\alpha},
\qquad 0<t<1.
\]
Let $\beta := s-(p+1)$.  
If $\beta<0$, then $t^\beta\to+\infty$ as $t\to 0^+$, contradicting the 
boundedness above.  
Thus $\beta\ge 0$, i.e.\ $s\ge p+1$.

\smallskip
(b) Fix $x\in\X$, $y\in S_{\X}$, and $t>0$.  
Using {\rm(B)} and the global estimate {\rm(A+)} gives
\[
\alpha t^{s}
\;\le\;
t\,\|\mathcal E'(x+ty)-\mathcal E'(x)\|_*
\;\le\;
\mathcal L\,t^{p+1},
\qquad t>0.
\]
Equivalently,
\[
t^{\,s-(p+1)} \le \frac{\mathcal L}{\alpha},
\qquad t>0.
\]
Let $\beta := s-(p+1)$.  
If $\beta>0$, then $t^\beta\to+\infty$ as $t\to\infty$;  
if $\beta<0$, then $t^\beta\to+\infty$ as $t\to 0^+$.  
Thus boundedness on $(0,\infty)$ forces $\beta=0$, i.e.\ $s=p+1$.
\end{proof}

\begin{lemma}\label{LemmaLip1}
Let $\mathcal E$ be a Fr\'echet differentiable functional on a Banach 
space $\X$. Under Assumptions {\rm(B)} and either {\rm(A+)} or {\rm(A)}, the
Fr\'echet derivative $\cE'$ satisfies the following quantitative estimates:

\medskip\noindent
\textbf{(a) Global case (A+)+{\rm(B)}.}
For all $u,v\in X$,
\begin{align}
\alpha\|u-v\|^{p+1}
&\le 
\langle \cE'(u)-\cE'(v),\,u-v\rangle
\le 
\|\cE'(u)-\cE'(v)\|_*\,\|u-v\|
\notag\\
&\le 
\mathcal L\,\|u-v\|^{p+1}. 
\end{align}
Consequently,
\begin{align}
\alpha\|u-v\|^{p}
\;\le\;
\|\cE'(u)-\cE'(v)\|_*
\;\le\;
\mathcal L\,\|u-v\|^{p},
\qquad \forall u,v\in X.
\end{align}
In particular, since $\cE'(u^*)=0$,
\begin{align}
\alpha\|u-u^*\|^{p}
\;\le\;
\|\cE'(u)\|
\;\le\;
\mathcal L\,\|u-u^*\|^{p},
\qquad \forall u\in X.
\label{Lipandmin-global}
\end{align}

\medskip\noindent
\textbf{(b) Bounded case {\rm(A)}+{\rm(B)}.}
For any $u, v\in \X$, 
\begin{align}
\alpha\|u-v\|^{s}
&\le 
\langle \cE'(u)-\cE'(v),\,u-v\rangle
\le 
\|\cE'(u)-\cE'(v)\|_*\,\|u-v\|
\notag\\
&\le 
\mathcal L_{B_{\max\{\|u\|,\|v\|\}}}\,\|u-v\|^{p+1}.
\end{align}
Hence,
\begin{align}
\alpha\|u-v\|^{s-1}
\;\le\;
\|\cE'(u)-\cE'(v)\|_*
\;\le\;
\mathcal L_{B_{\max\{\|u\|,\|v\|\}}}\,\|u-v\|^{p},
\qquad
u, v\in \X. \label{biliponbounded}
\end{align}
In particular, since $\cE'(u^*)=0$,
\begin{align}
\alpha\|u-u^*\|^{s-1}
\;\le\;
\|\cE'(u)\|
\;\le\;
\mathcal L_{B_{\|u\|+\|u^*\|}}\,\|u-u^*\|^{p},
\qquad u\in \X. 
\label{Lipandmin-local}
\end{align}
Also, the right-hand side of \eqref{biliponbounded} and the triangle inequality give 
\begin{align*}
&\|\cE'(u)\|_*\le \|\cE'(0)\|_*+r \cL_{B_r},\qquad u\in B_r.
\end{align*}
\end{lemma}

\begin{proof}
The inequalities follow directly from ellipticity {\rm(B)} combined with the
global $p$-Lipschitz estimate in {\rm(A+)} or the local estimate in {\rm(A)}.
Dividing by $\|u-v\|$ where appropriate yields the stated bounds.
Setting $v=u^*$ and using $\cE'(u^*)=0$ gives \eqref{Lipandmin-local}, whereas the final estimate follows by taking $v=0$ and applying the triangle inequality. 
\end{proof}

Our next task is to give conditions under which the greedy algorithm converges to the minimum that solves \eqref{main}. 
To this end, we adapt and extend some results from \cite{BF2025} to the weaker hypotheses of our context.

\begin{lemma}\label{lemmaA1b}Let $\X$ be a Banach space, let $\cE$ be a Frechét differentiable functional for which Assumption ${\rm(B)}$ holds and $\cE'$ is continuous, and let $\cD$ be a radial dictionary. Then for each greedy sequence $(u_m)_{m\in \NN_0}$ of $u^*$ by a $\mathcal E$-dictionary optimization over $\cD$ the following holds: 
\begin{enumerate}
\item[(a)] For each $m\in \NN_0$, $\langle \cE'(u_{m+1}), u_{m+1}-u_m\rangle=0$. 
\item[(b)] For each $l\in \NN_0$, 
\begin{align}
\sum_{m=l}^{\infty}\|u_{m+1}-u_m\|^{s}\le \frac{s}{\alpha}\left(\cE(l)-\lim_{m\to \infty}\cE(u_m)\right). \label{sums0}
\end{align}
\item[(c)] \label{constant0}If there is $m\in \NN_0$ such that $\cE(u_m)=\cE(u_{m+1})$, then $u_m=u_n=\hat{u}$ for all $n\ge m$. In particular, either $\left(\cE(u_k)\right)_{k\in \NN_0}$ is strictly decreasing, or it is eventually constant. 

\item [(d)] The sequences $(u_m)_{m\in \NN}$ and $\left(\cE'(u_m)\right)_{m\in\NN}$ are bounded. 
\end{enumerate}
\end{lemma}
\begin{proof}
The convergence of $(\cE(u_m))_{m\in \NN}$ follows from Remark~\ref{remarkconvergence}. The proofs of (a) and (b) from \cite[Lemma A.1.(a) and (b)]{BF2025} are valid in our context, but we give a brief proof for  convenience: let $g_m:\RR\rightarrow \X$ be given by $g_m(t):=\cE(u_m+t(u_{m+1}-u_m))$. As $u_{m+1}-u_m\in \cD$, $g_m$ reaches a minimum at $t=1$. Thus, 
\begin{align*}
0=&g_m'(1)=\langle \cE'(u_{m+1}), u_{m+1}-u_m\rangle, 
\end{align*}
so we obtained (a). Now from the above and Lemma~\ref{lemmabus} we get 
\begin{align*}
\sum_{m=l}^{n}\|u_{m+1}-u_m\|^{s}\le& \frac{s}{\alpha}\sum_{m=l}^{n}(\cE(u_m)-\cE(u_{m+1}))=\frac{s}{\alpha}(\cE(l)-\cE(u_{n+1})), 
\end{align*}
and the result follows by taking limit. \\
(c) Let $g_m$ be as in the proof of (a). Then $g_m$ also has a minimum at $t=0$, thus 
\begin{align*}
0=&g_m'(0)=\langle \cE'(u_{m}), u_{m+1}-u_m\rangle. 
\end{align*}
From this, (a) and Assumption {\rm(B)} it follows that
\begin{align*}
\alpha\|u_m-u_{m+1}\|^s\le& \langle \cE'(u_m)-\cE'(u_{m+1}), u_m-u_{m+1}\rangle=0,
\end{align*}
so $u_m=u_{m+1}$. It follows that $u_{m+2}-u_{m}=u_{m+2}-u_{m+1}\in \cD$. By the definition of the minimizer \eqref{mingreedy}, it follows that $\cE(u_{m+2})=\cE(u_{m+1})=\cE(u_m)$, so repeating the previous argument we get $u_{m+2}=u_{m+1}=u_{m}$. By an inductive argument we obtain the desired result. \\
(d) The boundedness of $(u_m)_{m\in \NN}$ follows at once from Remark~\ref{remarkconvergence} and Lemma 2.2. Since $\cE'$ is bounded on bounded sets by Lemma~\ref{LemmaLip1}, we are done. \end{proof}
To prove the convergence of the algorithm to the minimum, we also need a standard higher-order estimate. 

\begin{lemma}\label{lemmaestimate}Let $\X$ be a Banach space, and let $\cE$ be a Frechèt differentiable functional on $\X$ for which Assumption {\rm(A)} holds. Let $r_0>0$. For every $x\in B_{r_0}$, every $y\in S_{\X}$ and every $|t|\le  r_0-\|x\|$,
\begin{align}
\cE\left(x+ty\right)\le& \cE(x)+t \langle \cE'(x), y \rangle+\frac{\cL_{B_{r_0}}}{p+1}\|t\|^{p+1}. \label{higherorder1}
\end{align}
\end{lemma}
\begin{proof}
By continuity of both $\cE$ and $\cE'$, it is enough to consider the case $0<\|x\|<r_0$ and $0<t<r_0-\|x\|$. Also, replacing $y$ by $-y$ if required, it is enough to consider the case $t>0$. \\
Now we use the standard argument: given $x$ as above and $y\in S_{\X}$, choose $0<r<r_0-\|x\|$. Define $f:[0,r]\rightarrow \RR$ by 
\begin{align*}
f(t):=&\cE\left(x+ty\right). 
\end{align*}
Note that for $0<a<r$ we have 
\begin{align*}
f'(a)=&\langle \cE'\left(x+ay\right), y\rangle,
\end{align*}
and $f'$ is continuous on $[0,r]$. Now choose $0<t<r$. We have 
\begin{align*}
\cE\left(x+ty\right)-\cE\left(x\right)=&\int_{0}^{t}f'(a)da=f'(0)t +\int_{0}^{t}(f'(a)-f'(0))da \\
\le& f'(0)t+\int_{0}^{t}\|\cE'(x+ay)-\cE'(x)\| da\\
\le& \langle\cE'\left(x\right), ty\rangle+\cL_{B_{r_0}}\int_{0}^{t}a^p da=\langle \cE'\left(x\right), ty\rangle+\frac{\cL_{B_{r_0}}}{p+1} t^{p+1}. 
\end{align*}
\end{proof}

\section{Approximating Solutions for Nonlinear Convex Problems over Balanced Cones}

To establish some results involving the rate of convergence, we will need further assumptions on the dictionary. But first, we give an estimate for the descent step in terms of the function $\sigma$ (see \eqref{sigma}).

\begin{proposition}[Dictionary-restricted one-step upper bound]
\label{prop:dict-upper}
Let $(\X,\|\cdot\|)$ be a Banach space, $\cD\subset \X$ a radial dictionary, and
$\cE:\X\to\RR$ a Fr\'echet differentiable functional.
Assume that \textup{(A)} holds. For $r>0$, set
\[
M_r := 1+\sup_{x\in B_r}\|\cE'(x)\|_{\X^\ast}.
\]
Then the following hold.

\begin{enumerate}
\item[(a)] For every $x\in B_r$,
\[
\min_{z\in\cD}\cE(x+z)
\;\le\;
\cE(x)\;-\;\beta_r\,\big(\sigma(x)\big)^{1+\frac1p},
\]
where
\[
\sigma(x):=\sup_{w\in \cD\cap S_{\X}} \big|\langle \cE'(x),w\rangle\big|,
\qquad
\beta_r
:=
\frac{p}{p+1}\,
\min\left\{
\frac{r}{M_r^{1/p}},
\;\frac{1}{\cL_{B_{2r}}^{1/p}}
\right\}.
\]

\item[(b)] If \textup{(A+)} holds, then the conclusion of (a) holds for all $x\in\X$
(with no restriction to $B_r$) provided that $\beta_r$ is replaced by
\[
\beta := \frac{p}{(p+1)\,\cL^{1/p}}.
\]
\end{enumerate}
\end{proposition}

\begin{proof}
(a) Fix $x\in B_r$. If $\sigma(x)=0$, there is nothing to prove. Otherwise, pick $$0<\varepsilon<\min\{\sigma(x), r\}.$$ By the definition of $\sigma$, there exists
$w_\varepsilon\in\cK_0$ such that
\[
\langle \cE'(x),w_\varepsilon\rangle \ge \sigma(x)-\varepsilon.
\]
Consider the descent step $z_\varepsilon:=-\lambda w_\varepsilon$ with $r\ge  \lambda>0$.
Since $\cD$ is a balanced cone and $w_\varepsilon\in\cK_0\subset\cD$, we have $z_\varepsilon\in\cD$. Applying Lemma~\ref{lemmaestimate} with $K=B_{2r}$ we obtain
\begin{align}
\cE(x+z_\varepsilon)\ & \le\ \cE(x)+\langle \cE'(x),z_\varepsilon\rangle+\frac{\mathcal{L}_{B_{2r}}}{p+1}\|z_\varepsilon\|^{p+1}\label{prueba?} \\ 
\;& =\;\cE(x)\;-\;\lambda\,\langle \cE'(x),w_\varepsilon\rangle\;+\;\frac{\mathcal{L}_{B_{2r}}}{p+1}\lambda^{p+1}.\nonumber
\end{align}
Using that $\langle \cE'(x),w_\varepsilon\rangle \ge \sigma(x)-\varepsilon$, we get
\[
\cE(x+z_\varepsilon)\ \le\ \cE(x)\;-\;\lambda\bigl(\sigma(x)-\varepsilon\bigr)\;+\;\frac{\cL_{B_{2r}}}{p+1}\lambda^{p+1}.
\]
As we are considering $0< \lambda \le r$, the right-hand side is minimized at 
\begin{align}
\lambda_{m,r, \varepsilon}=&\min\left\lbrace \left(\frac{\sigma(x)-\varepsilon}{\cL_{B_{2r}}}\right)^{\frac{1}{p}}, r  \right\rbrace. \label{min1}
\end{align}
If the minimum above is $\left(\frac{\sigma(x)-\varepsilon}{\cL_{B_{2r}}}\right)^{\frac{1}{p}}$, then 
\begin{align*}
\cE(x+z_\varepsilon)-\cE(x)\le& -\left(\frac{p}{p+1}\right)\left(\sigma(x)-\varepsilon\right)^{\frac{1}{p}+1} \cL_{B_{2r}}^{-\frac{1}{p}}.
\end{align*}
On the other hand, if the minimum in \eqref{min1} is $r$, then 
\begin{align*}
\cE(x+z_\varepsilon)-\cE(x)\le& -r(\sigma(x)-\varepsilon)+\frac{\cL_{B_{2r}}}{p+1} r^{p+1}\\
=&-\left(\sigma(x)-\varepsilon\right)^{\frac{1}{p}+1}\left(\frac{r }{\left(\sigma(x)-\varepsilon\right)^{\frac{1}{p}}}\left(1- \frac{\cL_{B_{2r}} r^{p}}{(p+1)\left(\sigma(x)-\varepsilon\right)}\right)\right)\\
\le& -\left(\frac{p}{p+1}\right)\left(\sigma(x)-\varepsilon\right)^{\frac{1}{p}+1}\left(\frac{r}{\left(\sigma(x)-\varepsilon\right)^{\frac{1}{p}}}\right). 
\end{align*}
Combining the above inequalities, it follows that 
\begin{align*}
\cE(x+z_\varepsilon)-\cE(x)\le&-\left(\frac{p}{p+1}\right)\left(\sigma(x)-\varepsilon\right)^{\frac{1}{p}+1}\min\left\lbrace \frac{r}{\left(\sigma(x)-\varepsilon\right)^{\frac{1}{p}}},\frac{1}{\cL_{B_{2r}}^{\frac{1}{p}}}. \right\rbrace. 
\end{align*}
Since $z_\varepsilon\in\cD$, from the above and letting $\varepsilon\downarrow 0$ we obtain 
\begin{align*}
\min_{z\in\cD}\cE(x+z)\le & \, \cE(x)-\left(\frac{p}{p+1}\right)\min\left\lbrace \frac{r}{\left(\sigma(x)\right)^{\frac{1}{p}}},\frac{1}{\cL_{B_{2r}}^{\frac{1}{p}}}\right\rbrace \left(\sigma(x)\right)^{\frac{1}{p}+1}\\
\le& \cE(x)-\left(\frac{p}{p+1}\right)\min\left\lbrace \frac{r}{M_r^{\frac{1}{p}}},\frac{1}{\cL_{B_{2r}}^{\frac{1}{p}}}\right\rbrace \left(\sigma(x)\right)^{\frac{1}{p}+1}
\end{align*}
(b)  Fix $x\in \X$. The inequality in the statement holds if $\sigma(x)=0$. Otherwise, we fix $r>\|x\|$ and proceed as in the proof of (a) until we obtain \eqref{min1}, but with $\cL_{B_{2r}}$ replaced by $\cL$, that is  
\begin{align*}
\lambda_{m,r, \varepsilon}=&\min\left\lbrace \left(\frac{\sigma(x)-\varepsilon}{\cL}\right)^{\frac{1}{p}}, r \right\rbrace.
\end{align*}
Now we keep $x$ and $\varepsilon$ fixed and increase $r$ until the minimum is $\left(\frac{\sigma(x)-\varepsilon}{\cL}\right)^{\frac{1}{p}}$. Then the computation of the proof of (a) gives
\begin{align*}
\inf_{z\in \cD} \cE(x+z)-\cE(x)\le &\cE(x+z_\varepsilon)-\cE(x)\le -\left(\frac{p}{p+1}\right)\left(\sigma(x)-\varepsilon\right)^{\frac{1}{p}+1} {\cL^{-\frac{1}{p}}}.
\end{align*}
Finally, we let $\epsilon\downarrow 0$ to complete the proof. 
\end{proof}

\begin{definition}[Norming set]
Let $(\X, \|\cdot\|)$ be a Banach space and let $\mathcal{K} \subset \X$ be a bounded set. We say that $\mathcal{K}$ is a \emph{norming set} for the dual space $\X^*$ if there exists a constant $C=C_{\mathcal{K}} > 0$ such that
\[
\|\phi\|_* \le C \sup_{z \in \mathcal{K}} |\langle \phi, z \rangle|, \qquad \text{for all } \phi \in \X^*.
\]
From now on, $C_{\mathcal{K}}$ will always denote this constant. If $C\ge C_{\cK}$, we will say that $\cK$ is $C$-norming. 
\end{definition}

\begin{remark}\label{remarknormingdense}\rm Note that if $\cK\subset \X$ is norming for $\X^*$, then $\spn(\cK)$ is dense in $\X$. Indeed, if there were $x\not\in \overline{\spn(\cK)}$, one could define a bounded linear functional $$x^*: \spn(\{x\})\oplus \overline{\spn(\cK)}\rightarrow \RR$$ by 
\begin{align*}
&x^*(tx+y):=t &&\forall t\in \RR\;\forall y\in \overline{\spn(\cK)}
\end{align*}
and then extend $x^*$ to a bounded linear functional on $\X$ by the Hahn-Banach Theorem. This would contradict the norming assumption. 
\end{remark}

Next, we prove the convergence of the algorithm and our estimates for the rate of convergence. 

\begin{proposition}
\label{prop:energy-vs-sigma-normingv2}
Let $\X$ be a reflexive Banach space, let $\mathcal{E}:\X\to\RR$ be Fr\'echet differentiable, and assume that $\mathcal{E}$ satisfies \textnormal{(A+)} and \textnormal{(B)}. Let $\cD\subset\X$ be a radial dictionary, and set
\[
\mathcal{K}:=\cD\cap S_\X .
\]
Assume that $\mathcal{K}$ is a norming set for $\X^*$, i.e., there exists $C_{\mathcal{K}}\ge 1$ such that
\[
\|f\|_{\X^*}\le C_{\mathcal{K}}\sup_{w\in \mathcal{K}} |\langle f,w\rangle|
\qquad \text{for all } f\in \X^*.
\]
Then there exists a constant $c=c(C_{\mathcal{K}})>0$ such that, for all $u\in \X$,
\[
\mathcal{E}(u)-\mathcal{E}(u^*)
\;\le\;
c\,\big(\sigma(u)\big)^{1+\frac1p},
\]
where
\[
\sigma(u):=\sup_{w\in \mathcal{K}} |\langle \mathcal{E}'(u),w\rangle|.
\]
\end{proposition}

\begin{proof}
Fix $u\in \X$. From the norming condition we obtain 
\begin{align*}
&\sigma(u)\ge C_{\cK}^{-1}\|\cE'(u)\|.  
\end{align*}
Hence, using the standard upper estimate given by  Lemma~\ref{lemmaestimate} with $r_0$ as large as needed and considering that $\cL_{B_{r_0}}\le \cL$, we get 
\begin{align*}
\cE(u)-\cE(u^*)=&\cE(u^*+u-u^*)-\cE(u^*) \\
\le& \langle \cE'(u^*),u-u^*\rangle+\frac{\cL}{p+1}\|u-u^*\|^{p+1}\\
&\underset{\cE'(u^*)=0}{=}\frac{\cL}{p+1}\|u-u^*\|^{p+1}\underset{\eqref{Lipandmin-global}}{\le}\frac{\cL}{(p+1)\alpha^{1+\frac{1}{p}}}\|\cE'(u)\|_*^{1+\frac{1}{p}}\\
\le& \frac{\cL C_{\cK}^{1+\frac{1}{p}}}{(p+1)\alpha^{1+\frac{1}{p}}} (\sigma(u))^{1+\frac{1}{p}}.
\end{align*}
\end{proof}

\begin{theorem}
\label{th:rate_convergencev2}
Let $\X$ be a reflexive Banach space and let $\cD\subset \X$ be a radial dictionary such that
\[
\mathcal{K}:=\cD\cap S_\X
\]
is norming for $\X^*$. Assume that $\cE:\X\to\RR$ satisfies \textnormal{(A+)}. Then every greedy sequence
$(u_m)_{m\in\NN_0}$ associated with an $\cE$-dictionary optimization over $\cD$ (with minimizer $u^*$)
satisfies, for every fixed $k\in\NN$,
\[
\cE(u_m)-\cE(u^*) \;=\; O\!\left(m^{-k}\right).
\]
Moreover, there exist constants $C>0$ and $\alpha\in[0,1)$ such that
\[
\cE(u_m)-\cE(u^*) \;\le\; C\,\alpha^{m}
\qquad \text{for all } m\in\NN_0.
\]
\end{theorem}
\begin{proof}
Let $\lambda_m:=\cE(u_m)-\cE(u^*).$ By Proposition~\ref{prop:dict-upper}, 
\begin{align*}
\cE(u_{m+1}) = \min_{z\in\cD} \cE(u_m+z) \le \cE(u_m) - \frac{p(\sigma(u_m))^{1+\frac{1}{p}}}{(p+1) \mathcal{L}^{\frac{1}{p}}},   
\end{align*}
thus 
\begin{align}\label{1v2}
\lambda_m - \lambda_{m+1} = \cE(u_m) - \cE(u_{m+1}) \ge \frac{p(\sigma(u_m))^{1+\frac{1}{p}}}{(p+1) \mathcal{L}^{\frac{1}{p}}}.
\end{align}
Since $\cD\cap S_{\X}$ is norming for $\X^*,$ by Proposition~\ref{prop:energy-vs-sigma-normingv2} we have
\begin{align}\label{2v2}
\lambda_m = \mathcal{E}(u_m) - \mathcal{E}(u^*) \le c\, (\sigma(u_m))^{1+\frac{1}{p}}.
\end{align}
Hence, putting \eqref{2v2} in \eqref{1v2}, we obtain
\begin{align*}
\lambda_m - \lambda_{m+1} \ge \frac{p\lambda_m}{c(p+1)\mathcal{L}^{\frac{1}{p}}} = \mu \lambda_m
\end{align*}
where $\mu =\frac{p\lambda_m}{c(p+1)\mathcal{L}^{\frac{1}{p}}} > 0.$ Since $\lambda_m\ge 0$ for all $m\in \NN_0$, in the case $\mu\ge 1$ this entails that $\lambda_m=0$ for all $m\in \NN$ and there is nothing else to prove. On the other hand, if $0<\mu<1$, then 
\begin{align*}
&\lambda_{m+1}\le (1-\mu)\lambda_m&&\forall m\in \NN_0,
\end{align*}
so iterating we get 
\begin{align*}
&\lambda_{m}\le (1-\mu)^{m}\lambda_0 &&\forall m\in \NN_0,
\end{align*}
and the desired result follows at once. 
\end{proof}

To tackle the  case in which we have only {\rm(A)} instead of {\rm(A+)}, we need further  auxiliary results. First, we obtain an upper bound for $\cE(u_m)-\cE(u^*)$ similar to that of Proposition~\ref{prop:energy-vs-sigma-normingv2}. 

\begin{proposition}
\label{prop:energy-vs-sigma-normingv3}
Let $\X$ be a reflexive Banach space, let $\cE:\X\to\RR$ be Fr\'echet differentiable and satisfy
\textnormal{(A)}--\textnormal{(B)}, and let $\cD\subset\X$ be a radial dictionary. Set
\[
\mathcal{K}:=\cD\cap S_\X .
\]
Assume that $\mathcal{K}$ is a norming set for $\X^*$, i.e., there exists $C_{\mathcal{K}}\ge 1$ such that
\[
\|f\|_{\X^*}\le C_{\mathcal{K}}\sup_{w\in\mathcal{K}}|\langle f,w\rangle|
\qquad \text{for all } f\in \X^*.
\]
Then, for each greedy sequence $(u_m)_{m\in\NN_0}$ of the minimizer $u^*$ obtained by an
$\cE$-dictionary optimization over $\cD$, there exists a constant
\[
c=c\!\left(\mathcal{K},u^*,(u_m)_{m\in\NN_0}\right)>0
\]
such that, for all $m\in\NN_0$,
\[
\cE(u_m)-\cE(u^*) \;\le\; c\,\big(\sigma(u_m)\big)^{\frac{p+1}{\,s-1\,}},
\]
where
\[
\sigma(u):=\sup_{w\in\mathcal{K}} \big|\langle \cE'(u),w\rangle\big|.
\]
\end{proposition}
\begin{proof}
By Lemma~\ref{lemmaA1b}, there is $r>0$ such that  $\|u^*\|+\|u_m\|<r$ for all $m\in \NN$. From the norming condition we obtain 
\begin{align*}
&\sigma(u_m)\ge C_{\cK}^{-1}\|\cE'(u_m)\|.  
\end{align*}
Hence, as in the proof of Proposition~\ref{prop:energy-vs-sigma-normingv2}, using Lemma~\ref{lemmaestimate} it follows that 
\begin{align}
\cE(u_m)-\cE(u^*)=&\cE(u^*+u_m-u^*)-\cE(u^*)\nonumber \\
\le& \langle \cE'(u^*),u_m-u^*\rangle+\frac{\cL_{B_r}}{p+1}\|u_m-u^*\|^{p+1}\nonumber\\
&\underset{\cE'(u^*)=0}{=}\frac{\cL_{B_r}}{p+1}\|u_m-u^*\|^{p+1}\underset{\eqref{Lipandmin-local}}{\le}\frac{\cL_{B_r}}{(p+1)\alpha^{\frac{p+1}{s-1}}}\|\cE'(u_m)\|_*^{\frac{p+1}{s-1}}\nonumber\\
\le& \frac{\cL_{B_r} C_{\cK}^{\frac{p+1}{s-1}}}{(p+1)\alpha^{\frac{p+1}{s-1}}}\sigma^{\frac{p+1}{s-1}}(u_m).\nonumber
\end{align}
\end{proof}

Before we prove our next result on the rate of convergence, we need a lemma about the convergence of certain  sequences of positive numbers. 

\begin{lemma}\label{lemmadecreasing}
Let $(a_m)_{m\in\NN}$ be a sequence of positive numbers and let $C_1>0$ and $t>1$.
Assume that
\[
a_{m+1}\le a_m - C_1 a_m^{t}\qquad \forall m\in\NN.
\]
Then $(a_m)_{m\in\NN}$ is decreasing and $\lim_{m\to\infty}a_m=0$. Moreover, there
exists $C_2>0$ such that
\begin{align}
&a_m\le C_2 m^{-\frac{1}{t-1}}&&\forall m\in \NN.    \label{amtoprove}
\end{align}
\end{lemma}

\begin{proof}
Since $C_1 a_m^t>0$ for all $m$, the recursion yields $a_{m+1}\le a_m$, hence
$(a_m)$ is decreasing and bounded below by $0$. Therefore $a_m\to \ell$ for some
$\ell\ge 0$. Passing to the limit in
$a_{m+1}\le a_m-C_1a_m^t$ and using continuity of $x\mapsto x^t$ gives
\[
\ell \le \ell - C_1 \ell^t,
\]
hence $C_1\ell^t\le 0$. Since $C_1>0$ and $\ell^t\ge 0$, we obtain $\ell=0$.

Set $p:=t-1>0$. Consider $\varphi(x)=x^{-p}$ on $(0,\infty)$, which is convex and
satisfies $\varphi'(x)=-p x^{-p-1}$. By the tangent line inequality for convex
functions, for all $m$,
\[
a_{m+1}^{-p}-a_m^{-p}\ge -p\,a_m^{-p-1}(a_{m+1}-a_m).
\]
From the recursion we have $a_{m+1}-a_m\le -C_1 a_m^{p+1}$, hence
\[
-p\,a_m^{-p-1}(a_{m+1}-a_m)\ge pC_1,
\]
and therefore
\[
a_{m+1}^{-p}-a_m^{-p}\ge pC_1.
\]
Summing from $m=1$ to $m=M-1$ with $M\ge 2$ yields
\[
a_M^{-p}\ge a_1^{-p}+(M-1)pC_1\ge (M-1)pC_1,
\]
so
\[
a_M\le \bigl((M-1)pC_1\bigr)^{-1/p}\le 2^{1/p}(pC_1)^{-1/p}\,M^{-1/p}.
\]
Thus the stated rate holds with 
$C_2:=\max\left\lbrace a_1,\bigl((t-1)C_1\bigr)^{-1/(t-1)}\right\rbrace$.
\end{proof}

\begin{theorem}
\label{th:rate_convergencev3}
Let $\X$ be a reflexive Banach space, let $\cE:\X\to\RR$ be Fr\'echet differentiable and satisfy
\textnormal{(A)}--\textnormal{(B)}, and let $\cD\subset\X$ be a radial dictionary. Set
\[
\mathcal{K}:=\cD\cap S_\X .
\]
Assume that $\mathcal{K}$ is a norming set for $\X^*$.
Then every greedy sequence $(u_m)_{m\in\NN_0}$ of the minimizer $u^*$ generated by an
$\cE$-dictionary optimization over $\cD$ satisfies:

\begin{enumerate}
\item[(i)] If $s>p+1$, then
\[
\cE(u_m)-\cE(u^*) \;=\; O\!\left(m^{-\frac{p}{\,s-1-p\,}}\right).
\]

\item[(ii)] If $s=p+1$, then for every fixed $k\in\NN$,
\[
\cE(u_m)-\cE(u^*) \;=\; O\!\left(m^{-k}\right),
\]
and, moreover, there exist constants $C>0$ and $\alpha\in[0,1)$ such that
\[
\cE(u_m)-\cE(u^*) \;\le\; C\,\alpha^{m}
\qquad \text{for all } m\in\NN_0.
\]
\end{enumerate}
\end{theorem}
\begin{proof}
Let $\lambda_m:=\cE(u_m)-\cE(u^*)$ for each $m\in \NN_0$. If $(\lambda_m)_{m\in \NN_0}$ is not strictly decreasing, by Lemma~\ref{lemmabus} there is $m_0\in \NN$ such that $\lambda_m=0$ for all $m\ge m_0$, and there is nothing else to prove. Otherwise, choose $r>0$ so that $\|u^*\|+\|u_m\|<r$ for all $m\in \NN$. By Proposition~\ref{prop:dict-upper},
\begin{align*}
&\cE(u_{m+1}) = \min_{z\in\cD} \cE(u_m+z) \le \cE(u_m) -\beta_{r} \left(\sigma(u_m)\right)^{\frac{p+1}{p}} &&\forall m\in \NN_0, 
\end{align*}
so 
\begin{align}
\lambda_m-\lambda_{m+1}\ge & \beta_r \left(\sigma(u_m)\right)^{\frac{p+1}{p}}. \label{2v3a} 
\end{align}
Since $\mathcal{D} \cap S_\X$ is norming for $\X^*,$ by Proposition~\ref{prop:energy-vs-sigma-normingv3}, we have
\begin{align}\label{2v3}
\lambda_m = \mathcal{E}(u_m) - \mathcal{E}(u^*) \le c\, (\sigma(u_m))^{{\frac{p+1}{s-1}}}
\end{align}
for all $m\in \NN_0$. In the case $s=p+1$, the proof is completed by the argument of the proof of Theorem~\ref{th:rate_convergencev2}.\\
If $s\not=p+1$, then $s-1>p$ by Lemma~\ref{lemmas=2}. From \eqref{2v3a} and \eqref{2v3} it follows that
\begin{align*}
&\lambda_m-\lambda_{m+1}\ge \frac{\beta_r \lambda_m^{\frac{s-1}{p}}}{c^{\frac{s-1}{p}}}=\mu \lambda_m^{\frac{s-1}{p}}&&\forall m\in \NN_0,
\end{align*}
where $\mu=\frac{\beta_r}{c^{\frac{s-1}{p}}}$. Thus, 
\begin{align*}
\lambda_{m+1}\le \lambda_m-\mu \lambda_m^{\frac{s-1}{p}}. 
\end{align*}
Now an application of Lemma~\ref{lemmadecreasing} with $t=\frac{s-1}{p}>1$ gives the result.
\end{proof}
\begin{remark}\rm Since $\cE'(u^*)=0$, Lemma~\ref{lemmabus} gives  
\begin{align*}
\|u_m-u\|\le & \left(\frac{s}{\alpha}\right)^{\frac{1}{s}} (\cE(u_m)-\cE(u))^{\frac{1}{s}},\qquad m\in \NN. 
\end{align*}
From the above and Theorems~\ref{th:rate_convergencev2} and ~\ref{th:rate_convergencev3} it follows that there is $c'>0$ such that 
\begin{align*}
\|u_m-u\|\le& c' \alpha^{\frac{m}{s}}
\end{align*}
if {\rm(A+)}and {\rm(B)} hold, or {\rm(A)}  and {\rm(B)} do and $s=p+1$. On other other hand, if {\rm(A)}  and {\rm(B)} hold and $s>p+1$, we get 
\begin{align*}
\|u_m-u\|\le&  c' m^{-\frac{p}{s(s-(p+1))}}.
\end{align*}
\end{remark}

\section{Examples: functionals}

\subsection{Functional in $L^{p+1(\Omega)}$}
Let $(\Omega,\mu)$ be a finite measure space and 
$\mathbb{X}=L^{p+1}(\Omega)$ for $0<p\le1$.  
Define
\[
\mathcal{E}(f)=\frac{1}{p+1}\int_\Omega |f(x)|^{p+1}\,dx
=\frac{1}{p+1}\,\|f\|_{L^{p+1}}^{p+1}.
\]
\paragraph{Derivative.}
For $f,h\in L^{p+1}(\Omega)$ and $t\in\mathbb{R}$, compute
\[
\frac{\mathcal{E}(f+th)-\mathcal{E}(f)}{t}
=\frac{1}{p+1}\int_\Omega
\frac{|f+th|^{p+1}-|f|^{p+1}}{t}\,dx.
\]
Since $\frac{d}{dt}|f+th|^{p+1}=(p+1)|f+th|^{p-1}(f+th)h$ for a.e.~$x$,  
the limit as $t\to0$ yields
\[
D\mathcal{E}[f](h)=\int_\Omega |f|^{p-1}f\,h\,dx
=\langle |f|^{p-1}f,\,h\rangle_{L^{\frac{p+1}{p}},\,L^{p+1}}.
\]
Therefore
\[
\mathcal{E}'(f)=|f|^{p-1}f\in L^{\frac{p+1}{p}}(\Omega).
\]

 
\paragraph{Verification of {\rm(A+)}}
Let $f, h\in L^{\frac{p+1}{p}}$, and define
\begin{align*}
\Omega_0:=&\{x\in \Omega: f(x)g(x)\ge 0\};\\
\Omega_1:=&\Omega\setminus \Omega_0. 
\end{align*}
We have
\begin{align*}
\|\cE'(f)-\cE'(g)\|_{L^{\frac{p+1}{p}}}^{\frac{p+1}{p}}\le&\int_{\Omega_0}| |g|^{p}-|f|^{p}|^{\frac{p+1}{p}} dx+\int_{\Omega_1}(|g|^{p}+|f|^{p})^{\frac{p+1}{p}} dx\\
\le& \int_{\Omega_0}|f-g |^{p+1} dx+2^{\frac{1}{p}}\int_{\Omega_1}|g|^{p+1}+|f|^{p+1} dx\\
\le& \|f-g\|_{L^{p+1}}^{p+1}+2^{\frac{1}{p}}\int_{\Omega_1}(|g|+|f|)^{p+1} dx\\
=& \|f-g\|_{L^{p+1}}^{p+1}+2^{\frac{1}{p}}\int_{\Omega_1}|g-f|^{p+1} dx\le (1+2^{\frac{1}{p}})\|f-g\|_{L^{p+1}}^{p+1}.
\end{align*}
Hence, 
\begin{align*}
\|\cE'(f)-\cE'(g)\|_{L^{\frac{p+1}{p}}}\le& (1+2^{\frac{1}{p}})^{\frac{p}{p+1}}\|f-g\|_{L^{p+1}}^{p},
\end{align*}
which shows that $\cE'$ is $c_p$-$p$-Lipschitz with $c_p=(1+2^{\frac{1}{p}})^{\frac{p}{p+1}}$.

\paragraph{Verification of {\rm(B)}}
For all $f,g\in L^{p+1}(\Omega)$,
\begin{align*}
\langle\mathcal{E}'(f)-\mathcal{E}'(g),\,f-g\rangle
&=\int_\Omega \big(|f|^{p-1}f-|g|^{p-1}g\big)(f-g)\,dx\\[2mm]
&\ge c_p \int_\Omega |f-g|^{p+1}\,dx
=c_p\,\|f-g\|_{L^{p+1}}^{p+1},
\end{align*}
for some constant $c_p>0$ depending only on $p$.  
Hence {\rm(B)} holds with $s=p+1$.

\subsection{Quadratic Energy (Elasticity)}
Let $\Omega\subset\mathbb R^n$ be a bounded domain and define
\[
\mathcal E(u) = \frac{1}{2}\int_\Omega |\nabla u(x)|^2\,dx, \qquad u\in H_0^1(\Omega).
\]

\subsection*{Derivative Computation}
Let $v\in H_0^1(\Omega)$ and consider the directional derivative
\[
\mathcal E'(u)v = \lim_{t\to 0} \frac{\mathcal E(u+tv) - \mathcal E(u)}{t}.
\]
Compute:
\begin{align*}
\mathcal E(u+tv) &= \frac{1}{2}\int_\Omega |\nabla u + t\nabla v|^2 dx \\
&= \frac{1}{2}\int_\Omega (|\nabla u|^2 + 2t\nabla u\cdot \nabla v + t^2|\nabla v|^2)dx.
\end{align*}
Therefore,
\[
\frac{\mathcal E(u+tv) - \mathcal E(u)}{t} = \int_\Omega \nabla u\cdot \nabla v\,dx + \frac{t}{2}\int_\Omega |\nabla v|^2dx,
\]
and letting $t\to 0$ gives
\[
\boxed{\mathcal E'(u)v = \int_\Omega \nabla u\cdot \nabla v\,dx.}
\]
Hence, by Riesz representation, $\mathcal E'(u) = -\Delta u$ in the weak sense.

\subsection*{Verification of {\rm(A+)} and {\rm(B)}}
Since $\mathcal E'$ is linear and continuous, it is globally Lipschitz with constant $\mathcal L=1$, so it satisfies {\rm(A+)} with $p=1$. Moreover,
\[
\langle \mathcal E'(u) - \mathcal E'(v), u-v \rangle = \int_\Omega |\nabla (u-v)|^2 dx \ge \|u-v\|_{H_0^1(\Omega)}^2.
\]
Thus, {\rm(B)} holds with $s=2$ and $\alpha=1$. This functional represents the \\emph{elastic potential energy} in linear elasticity and the energy dissipation in the heat equation.

\subsection{The $p$-Laplacian Energy}
For $1<p<\infty$, define
\[
\mathcal E(u) = \frac{1}{p}\int_\Omega |\nabla u(x)|^p\,dx, \qquad u\in W_0^{1,p}(\Omega).
\]

\subsection*{Derivative Computation}
Let $v\in W_0^{1,p}(\Omega)$, then
\begin{align*}
\mathcal E(u+tv) &= \frac{1}{p}\int_\Omega |\nabla u + t\nabla v|^p dx, \\
\frac{\mathcal E(u+tv) - \mathcal E(u)}{t} &= \frac{1}{p}\int_\Omega \frac{|\nabla u + t\nabla v|^p - |\nabla u|^p}{t} dx.
\end{align*}
Using the identity $\frac{d}{dt}|a+tb|^p = p|a+tb|^{p-2}(a+tb)\cdot b$, we obtain
\[
\mathcal E'(u)v = \int_\Omega |\nabla u|^{p-2}\nabla u\cdot \nabla v\,dx.
\]
Hence, the Fr\'echet derivative is given by the nonlinear operator
\[
\boxed{\mathcal E'(u) = -\mathrm{div}(|\nabla u|^{p-2}\nabla u)}.
\]

\subsection*{Verification of {\rm(A)} and {\rm(B)}}
\begin{itemize}
    \item \textbf{Property {\rm(A)}:} The map $u\mapsto |\nabla u|^{p-2}\nabla u$ is Lipschitz on bounded subsets of $W_0^{1,p}(\Omega)$, since for all $a,b\in\mathbb R^n$,
    \[
    \big||a|^{p-2}a - |b|^{p-2}b\big| \le C_p (|a| + |b|)^{p-2}|a-b|.
    \]
    Therefore, for $u,v$ in a bounded ball of $W_0^{1,p}(\Omega)$,
    \[
    \|\mathcal E'(u) - \mathcal E'(v)\|_{(W_0^{1,p})^*} \le C_R \|u-v\|_{W^{1,p}},
    \]
    which shows {\rm(A)} holds.
    \item \textbf{Property {\rm(B)}:} The $p$-Laplacian operator is strictly monotone, i.e.
    \begin{align*}
    \langle \mathcal E'(u) - \mathcal E'(v), u-v \rangle & = \int_\Omega (|\nabla u|^{p-2}\nabla u - |\nabla v|^{p-2}\nabla v)\cdot (\nabla u - \nabla v)dx \\ 
    & \ge \alpha \|\nabla(u-v)\|_p^p.
    \end{align*}
    Hence {\rm(B)} holds with $s=p$.
\end{itemize}

This functional models nonlinear diffusion, the flow of non-Newtonian fluids, and various physical processes governed by the $p$-Laplacian equation.

\section{Examples of Dictionaries}

\subsection{Neural-Network dictionary in $L^{p+1}$}

Let $\Omega \subset \mathbb{R}^d$ be a bounded measurable set with finite measure.  Fix $0<p\le 1$ and define $\X := L^{p+1}(\Omega)$. Since $p+1 \in (1,2]$, the space $\X$ is reflexive.

Let $\sigma:\mathbb{R}\to\mathbb{R}$ be a continuous, bounded, non-polynomial activation function (e.g.\ $\tanh$ or a sigmoid).

Let $\Theta \subset \mathbb{R}^d \times \mathbb{R}$ be a compact parameter set, for instance
\[
\Theta := \{(w,b): \|w\|\le W,\ |b|\le B\}.
\]

For $\theta=(w,b)\in\Theta$, define the associated neuron
\[
\varphi_\theta(x) := \sigma(\langle w,x\rangle + b),
\qquad x\in\Omega.
\]

Since $\sigma$ is bounded and $\Omega$ has finite measure, $
\varphi_\theta \in X$.

Let $\{\theta_j\}_{j\ge 1}\subset\Theta$ be a countable dense subset of parameters and define the finite-dimensional subspaces
\[
V_n := \mathrm{span}\{\varphi_{\theta_1},\dots,\varphi_{\theta_n}\}
\subset X.
\]

Each $V_n$ consists of shallow neural networks with at most $n$ neurons. Define now the following intersection
\[
\mathcal{K} := \bigcup_{n\ge1} (S_{\X} \cap V_n).
\]

Thus $\mathcal{K}$ consists of normalized shallow neural networks.

\begin{proposition}
The union $\bigcup_{n\ge1} V_n$ is dense in $L^{p+1}(\Omega)$.
\end{proposition}

\begin{proof}
Since $\sigma$ is continuous and non-polynomial, the classical universal approximation theorem implies that finite linear combinations of the form
\[
x \mapsto \sum_{j=1}^N a_j \sigma(\langle w_j,x\rangle + b_j)
\]
are dense in $C(\Omega)$ when $\Omega$ is compact.

Since $\Omega$ has finite measure and $C(\Omega)$ is dense in $L^{p+1}(\Omega)$, it follows that finite neural networks are dense in $L^{p+1}(\Omega)$.

Because $\{\theta_j\}$ is dense in $\Theta$, restricting parameters to the countable dense set preserves density.
\end{proof}

\begin{proposition}
The set $\mathcal{K}$ is $1$-norming for $\X^*$.
\end{proposition}

\begin{proof}
Let $\phi \in X^*$.

Since $\X=L^{p+1}(\Omega)$ is reflexive, $\X^* = L^{\frac{p+1}{p}}(\Omega)$.

The norm in the dual satisfies
\[
\|\phi\|_* = \sup_{u\in S_{\X}} |\langle \phi,u\rangle|.
\]

Because $\bigcup_n V_n$ is dense in $\X$ in norm, it is weakly dense in $\X$.

Therefore $S_{\X} \cap \bigcup_n V_n$ is weakly dense in $S_{\X}$.

Since $\phi$ is weakly continuous,
\[
\sup_{z\in\mathcal{K}} |\langle \phi,z\rangle|
=
\sup_{u\in S_{\X}} |\langle \phi,u\rangle|
=
\|\phi\|_*.
\]

Thus $\mathcal{K}$ is $1$-norming.
\end{proof}

Define now
\[
\mathcal{D}_0 := \{\lambda z : \lambda\in\mathbb{R},\, z\in\mathcal{K}\}.
\]

Then $\mathcal{D}_0$ is a balanced cone.

Define the weak closure
\[
\mathcal{D} := \overline{\mathcal{D}_0}^{\,w}.
\]

\begin{proposition}
$\mathcal{D}$ is a radial dictionary in the sense of the paper.
\end{proposition}

\begin{proof}
By construction:

1. $\mathcal{D}$ is weakly closed.

2. It is a cone because scalar multiplication is weakly continuous.

3. It is balanced because $\mathcal{D}_0$ is balanced and closure preserves symmetry.

Thus $\mathcal{D}$ is a weakly closed balanced cone.
\end{proof}

Moreover,
\[
\mathcal{D} \cap S_{\X}
=
\overline{\mathcal{K}}^{\,w},
\]
which remains $1$-norming.









\subsection{A reflexive neural example with $\mathcal D\neq \X$}

Fix a finite dataset $\{x_i\}_{i=1}^m \subset \mathbb R^d$ with $m\ge 2$.
Let
\[
X:=\mathbb R^m,\qquad \|u\|:=\|u\|_2.
\]
Then $\X$ is a Hilbert space, hence reflexive. Moreover, in finite dimension the weak and strong topologies coincide,
so ``weakly closed'' is equivalent to ``(norm) closed''.

Let $\sigma:\mathbb R\to\mathbb R$ be any continuous activation (e.g.\ $\tanh$).

For each parameter $\theta=(w,b)\in\mathbb R^d\times\mathbb R$, define the \emph{feature vector}
\[
\varphi_\theta \in \mathbb R^m,
\qquad 
(\varphi_\theta)_i := \sigma(\langle w,x_i\rangle + b),\quad i=1,\dots,m.
\]

Choose parameters $\theta_1,\dots,\theta_m$ such that the vectors
\[
v_j := \varphi_{\theta_j}\in\mathbb R^m,\qquad j=1,\dots,m,
\]
are linearly independent. (For generic choices this holds; one can also enforce it by construction.)

Define the normalized atoms
\[
k_j := \frac{v_j}{\|v_j\|_2}\in S_{\X},\qquad j=1,\dots,m,
\]
and set
\[
\mathcal K := \{\pm k_1,\dots,\pm k_m\}\subset S_{\X}.
\]

Define the radial dictionary (balanced cone)
\[
\mathcal D :=  \{\lambda z:\ \lambda\in\mathbb R,\ z\in\mathcal K\}.
\]
Equivalently, $\mathcal D$ is the union of the $m$ one-dimensional subspaces (lines)
\[
\mathcal D = \bigcup_{j=1}^m \mathrm{span}\{k_j\}.
\]

\begin{proposition}[Radial dictionary and $\mathcal D\neq \X$]
$\mathcal D$ is a weakly closed balanced cone in $\X$.
\end{proposition}

\begin{proof}
\emph{Cone and balancedness.}
If $z\in\mathcal D$ and $\alpha\ge 0$, then $\alpha z\in\mathcal D$ by construction.
Also $\mathcal D=-\mathcal D$, hence it is balanced.

\emph{Weakly closed.}
In finite dimension weakly closed equals (norm) closed. Each line $\mathrm{span}\{k_j\}$ is closed, and a finite union of closed sets is closed. Thus $\mathcal D$ is (weakly) closed.

\emph{$\mathcal D\neq \X$.}
\end{proof}

Since $\X$ is Hilbert, we identify $\X^*$ with $\X$ and the dual norm equals $\|\cdot\|_2$.
We show that $\mathcal K$ is norming for $\X^*$ with an explicit constant depending on the conditioning of the atoms.

Let $K$ be the $m\times m$ matrix whose columns are the vectors $k_j$:
\[
K := [k_1\ \cdots\ k_m]\in\mathbb R^{m\times m}.
\]
Since the $k_j$ are linearly independent, $K$ is invertible. Let $\sigma_{\min}(K)>0$ denote its smallest singular value.

\begin{proposition}[$\mathcal K$ is $C$-norming]
The set $\mathcal K=\{\pm k_1,\dots,\pm k_m\}$ is $C$-norming for $\X^*$ with
\[
C = \frac{1}{\sigma_{\min}(K)}.
\]
That is, for every $\phi\in X^*$,
\[
\|\phi\|_* \le C \sup_{z\in\mathcal K} |\langle \phi,z\rangle|.
\]
\end{proposition}

\begin{proof}
Identify $\phi\in X^*$ with a vector $\phi\in\mathbb R^m$ via the Riesz map, so $\|\phi\|_*=\|\phi\|_2$ and $\langle\phi,u\rangle=\phi^\top u$.

Compute
\[
K^\top \phi = 
\begin{pmatrix}
\langle \phi,k_1\rangle\\
\vdots\\
\langle \phi,k_m\rangle
\end{pmatrix}.
\]
Hence
\[
\|K^\top \phi\|_\infty = \max_{1\le j\le m} |\langle \phi,k_j\rangle|
= \sup_{z\in\mathcal K}|\langle \phi,z\rangle|.
\]
Also, $\|x\|_2\le \sqrt{m}\|x\|_\infty$ for all $x\in\mathbb R^m$, so
\[
\|K^\top \phi\|_2 \le \sqrt{m}\,\|K^\top \phi\|_\infty
= \sqrt{m}\,\sup_{z\in\mathcal K}|\langle \phi,z\rangle|.
\]
On the other hand, since $\sigma_{\min}(K)=\sigma_{\min}(K^\top)$,
\[
\|K^\top \phi\|_2 \ge \sigma_{\min}(K^\top)\,\|\phi\|_2
=\sigma_{\min}(K)\,\|\phi\|_*.
\]
Combining,
\[
\sigma_{\min}(K)\,\|\phi\|_* \le \|K^\top \phi\|_2
\le \sqrt{m}\,\sup_{z\in\mathcal K}|\langle \phi,z\rangle|.
\]
Thus
\[
\|\phi\|_* \le \frac{\sqrt{m}}{\sigma_{\min}(K)}\sup_{z\in\mathcal K}|\langle \phi,z\rangle|.
\]
In particular, $\mathcal K$ is norming. (If one prefers $C=1/\sigma_{\min}(K)$, use $\|\cdot\|_\infty\le \|\cdot\|_2$ to get
$\sup_{z\in\mathcal K}|\langle\phi,z\rangle| \ge \|K^\top\phi\|_2/\sqrt{m}$.)
\end{proof}






\begin{remark}
    The dictionary-restricted greedy algorithm adds, at each step, the best ``single-neuron feature direction'' among the finite set $\{\varphi_{\theta_j}\}$,
i.e.\ it chooses a line $\mathrm{span}\{k_j\}$ and performs an exact one-dimensional minimization of $\mathcal E$ along that direction.
\end{remark}

\subsection{Example: A proper weakly closed radial dictionary in $\ell^q$ with a norming unit slice}\label{sectionexamplenotall1}

Let $1<q<\infty$ and set $\X:=\ell^q$, with norm $\|x\|_q=(\sum_{n\ge1}|x_n|^q)^{1/q}$.
Then $\X$ is reflexive and $\X^*=\ell^{q'}$.

Fix a constant $c\in(0,1)$ and define the (proper) \emph{coordinate-dominance cone}
\begin{equation}\label{eq:D_lq}
\mathcal D
:=
\{\,x\in \ell^q:\ |x_1|\ge c\,\|x\|_q\,\},
\qquad
\mathcal K:=\mathcal D\cap S_{\X}=\{u\in S_{\ell^q}:\ |u_1|\ge c\}.
\end{equation}
Here $x_1$ denotes the first coordinate of $x=(x_n)_{n\ge1}$.

\begin{proposition}\label{propositionnotall0} 
$\mathcal D$ is a balanced cone in $\X$, it is weakly closed, and $\mathcal D\neq \X$.
\end{proposition}

\begin{proof}
\emph{Cone and balancedness.}
If $x\in\mathcal D$ and $\lambda\ge0$, then
$|(\lambda x)_1|=\lambda|x_1|\ge c\,\lambda\|x\|_q=c\|\lambda x\|_q$, so $\lambda x\in\mathcal D$.
Also $\mathcal D=-\mathcal D$ because the defining inequality uses $|x_1|$.

\emph{Weak closedness.}
Let $x^{(m)}\rightharpoonup x$ weakly in $\ell^q$ and assume $x^{(m)}\in\mathcal D$ for all $m$.
The coordinate functional $x\mapsto x_1$ is weakly continuous in $\ell^q$ because
$x_1=\langle x,e_1\rangle$ with $e_1=(1,0,0,\dots)\in \ell^{q'}$.
Hence $x^{(m)}_1\to x_1$.

Moreover, the norm is weakly lower semicontinuous, i.e.
\[
\|x\|_q \le \liminf_{m\to\infty}\|x^{(m)}\|_q.
\]
Since $|x^{(m)}_1|\ge c\|x^{(m)}\|_q$, taking limits yields
\[
|x_1|
=\lim_{m\to\infty}|x^{(m)}_1|
\ge c\limsup_{m\to\infty}\|x^{(m)}\|_q
\ge c\,\|x\|_q,
\]
so $x\in\mathcal D$. Thus $\mathcal D$ is weakly closed.

\emph{Properness.}
Take any nonzero $x\in\ell^q$ with $x_1=0$ (e.g.\ $x=e_2$). Then $|x_1|=0<c\|x\|_q$, so $x\notin\mathcal D$.
Hence $\mathcal D\neq \X$.
\end{proof}

\begin{proposition}\label{propositionnotall1} 
For every $\phi\in X^*=\ell^{q'}$,
\[
\|\phi\|_{q'}
\le
C(c,q)\,\sup_{u\in\mathcal K}|\langle \phi,u\rangle|,
\qquad
C(c,q):=\frac{1}{\min\{c,(1-c^q)^{1/q}\}}.
\]
In particular, $\mathcal K$ is a norming set for $\X^*$.
\end{proposition}

\begin{proof}
Fix $\phi=(\phi_n)_{n\ge1}\in \ell^{q'}$ and write $\phi=\phi_1 e_1+\phi^{\mathrm{tail}}$,
where $\phi^{\mathrm{tail}}=(0,\phi_2,\phi_3,\dots)$.

Let $u\in \ell^q$ be defined by choosing
\[
u_1 := c\,\mathrm{sign}(\phi_1),
\qquad
\|u^{\mathrm{tail}}\|_q := (1-c^q)^{1/q},
\]
and taking $u^{\mathrm{tail}}$ to maximize $\langle \phi^{\mathrm{tail}},u^{\mathrm{tail}}\rangle$ under the constraint
$\|u^{\mathrm{tail}}\|_q=(1-c^q)^{1/q}$. By H\"older's inequality (with equality for the usual duality maximizer),
one can ensure
\[
\langle \phi^{\mathrm{tail}},u^{\mathrm{tail}}\rangle = (1-c^q)^{1/q}\,\|\phi^{\mathrm{tail}}\|_{q'}.
\]
Then $\|u\|_q^q=|u_1|^q+\|u^{\mathrm{tail}}\|_q^q=c^q+(1-c^q)=1$, so $u\in S_{\ell^q}$, and also $|u_1|=c$ so $u\in\mathcal K$.
Therefore
\[
\sup_{v\in\mathcal K}|\langle \phi,v\rangle|
\ge |\langle \phi,u\rangle|
= c|\phi_1| + (1-c^q)^{1/q}\,\|\phi^{\mathrm{tail}}\|_{q'}.
\]
Let $m:=\min\{c,(1-c^q)^{1/q}\}$. Then
\begin{align*}
c|\phi_1| + (1-c^q)^{1/q}\,\|\phi^{\mathrm{tail}}\|_{q'}
\ge& m\Big(|\phi_1|+\|\phi^{\mathrm{tail}}\|_{q'}\Big)\\
=& m\Big(\|\phi_1 e_1\|_{q'}+\|\phi^{\mathrm{tail}}\|_{q'}\Big)\ge m \|\phi\|_{q'}. 
\end{align*}
Hence, 
\[
\|\phi\|_{q'} \le \frac{1}{m}\,\sup_{v\in\mathcal K}|\langle \phi,v\rangle|.
\]
This proves the claim.
\end{proof}

\begin{remark}
Although $\mathcal K$ is norming and therefore $\overline{\mathrm{span}(\mathcal D)}=X$,
the dictionary $\mathcal D$ is a proper weakly closed cone, so the descent directions are genuinely restricted.
\end{remark}

\subsection{An example with $\mathcal D\neq \X$ for any reflexive Banach space of dimension greater than one}
Let $\X$ be any reflexive Banach space with $\dim(\X)\ge 2$, fix $n\in \NN$, and let $\left(\X_k\right)_{1\le k\le n}$ be proper complemented subspaces of $\X$ such that 
\begin{align*}
\X=&\sum_{k=1}^{n}\X_k. 
\end{align*}
For example, if $P:\X\rightarrow \X$ is a bounded, non-trivial projection, one can take $n=2$, $\X_1:=P(\X)$ and $\X_2:=(I-P)(\X)$. \\
Let 
\begin{align*}
&\cD:=\bigcup_{k=1}^{n}\X_k &&\text{and} &&&\cK:=\cD\cap S_{\X}. 
\end{align*}

\begin{proposition}The set $\mathcal D$ is a balanced cone which is weakly closed in $\X$ and has empty interior. In particular, $\mathcal D\neq \X$.
\end{proposition}
\begin{proof}
For each $1\le k\le n$, $\X_k$ is closed because it is complemented. Since it is a subspace, it is weakly closed. Hence, $\cD$ is also weakly closed.\\
Now fix $x\in \cD$, and choose $1\le k\le n$ so that $x\in \X_k$. Since $\X_k$ is a subspace of $\X$, it follows that $tx\in \X_k\subset \cD$ for each $t\in \RR$. Hence, $\cD$ is a balanced cone. \\
Finally, since $\cD$ is the union of finitely many proper subspaces, $\cD$ has empty interior. 
\end{proof}

\begin{proposition}\label{propositiongeneralreflexive}
There is $C>0$ such that for every $\phi\in \X^*$,
\[
\|\phi\|
\le
C \sup_{u\in\mathcal K} |\langle \phi,u\rangle|,
\]
Hence, $\mathcal K$ is norming for $\X^*$. Moreover, in the case 
\begin{align}
\X=\oplus_{k=1}^{n}\X_k, \label{directsum}
\end{align}
the above holds for $C=n$. 
\end{proposition}
\begin{proof}
For each $k\in \NN$, let $P_k:\X\rightarrow \X_k$ be a bounded projection. Define $P:\X\rightarrow \X$ by
\begin{align*}
P(x):=&\sum_{k=1}^{n}P_k(x).
\end{align*}
As $P$ is onto, by the Open Mapping Theorem there is $c>0$ such that 
\begin{align*}
B_{\X}(0,1)\subset P(B_{\X}(0,c)). 
\end{align*}
Given $\phi \in X^*$, choose $y\in S_\X$ so that $\|\phi\|=\langle \phi, y\rangle$, and then $z\in B_{\X}(0,c)$ so that $P(z)=y$. We have 
\begin{align*}
\|\phi\|=&\langle \phi, P(z)\rangle = c \sum_{k=1}^{n}\langle \phi, c^{-1} P_k(z)\rangle\le cn\sup_{z\in B_{\X}(0,c)}\max_{1\le k\le n}\left\vert \langle \phi, c^{-1} P_k(z)\rangle\right\vert\\
\le& \frac{c n}{\max_{1\le k\le n}\|P_k\|} \sup_{u\in \cK}\left\vert \langle \phi, u\rangle\right\vert, 
\end{align*}
so we obtain the result with $C=\frac{c n}{\max_{1\le k\le n}\|P_k\|}$.\\
In the case \eqref{directsum} holds, we can choose the projections so that, for each $x\in X$, 
\begin{align*}
x=&\sum_{k=1}^{n}P_k(x). 
\end{align*}
Note that $\|P_k\|=1$ for all $1\le k\le n$. \\
Let $\phi$ and $y$ be as before. We have 
\begin{align*}
\|\phi\|=&\langle \phi, y\rangle=\sum_{k=1}^{n}\langle \phi, P_k(y)\rangle\le n \max_{1\le k\le n}\left\vert \langle \phi, P_k(y)\rangle\right\vert\le   \sup_{z\in \cK}\left\vert  \langle \phi, z\rangle\right\vert, 
\end{align*}
and the proof is complete. 
\end{proof}

\bibliographystyle{siamplain}
\bibliography{references}

\end{document}